\documentclass{article}
\usepackage[margin=1in]{geometry}
\usepackage{subfigure}
\usepackage{diagbox}
\usepackage{tabularx}
\usepackage{authblk}
\usepackage{tikz}
\usetikzlibrary{matrix} 
\usetikzlibrary{arrows,automata}
\usetikzlibrary{positioning}

\usepackage{amssymb,mathrsfs,amsmath,amsthm,bm,diagbox,float}
\usepackage{multirow}
\usetikzlibrary{matrix,arrows}

\usepackage{mathtools}

\renewcommand{\S}{\mathcal{S}}

\newcommand{\B}{\mathcal{B}}
\newcommand{\A}{\mathcal{A}}

\newtheorem{theorem}{Theorem}[section]

\newtheorem{corollary}[theorem]{Corollary}
\newtheorem{lemma}[theorem]{Lemma}
\theoremstyle{definition}

\newtheorem{example}[theorem]{Example}
\newtheorem{remark}[theorem]{Remark}

		\title{Pattern-restricted cyclic permutations with a pattern-restricted cycle form}
		\author[1]{Kassie Archer\thanks{Department of Mathematics, United States Naval Academy, Annapolis, MD, 21402, karcher@usna.edu}, Ethan Borsh\thanks{Mathematics Department, Allegheny College, Meadville, PA, 16335, borsh01@allegheny.edu}, Jensen Bridges\thanks{Department of Mathematics, University of Florida, Gainesville, FL, 32611, jensenbridges@ufl.edu}, Christina Graves\thanks{University of Texas at Tyler, Tyler, TX, 75799, cgraves@uttyler.edu}, and Millie Jeske\thanks{University of Texas at Tyler, Tyler, TX, 75799, mjeske3@patriots.uttyler.edu}}
		\date{}
		
\begin{document}

\maketitle

\begin{abstract}
	In this paper, we consider cyclic permutations that avoid the monotone decreasing permutation $k(k-1)\ldots 21$, whose cycle also demonstrates some pattern avoidance. If the cycle is written in standard form with 1 appearing at the beginning of the cycle and the standard form avoids a pattern of length 3, we find answers in terms of continued fraction generating functions. We also consider the case that every cyclic rotation of the cycle form of the permutation avoids a pattern of length 4 and enumerate two such cases.
\end{abstract}

\section{Introduction and background}

The symmetric group $\S_n$ consists of permutations (i.e., bijections) of the set $[n]=\{1,2,\ldots,n\}$. These permutations can be written either in their cycle form as a product of disjoint cycles or in one-line notation, as $\pi = \pi_1\pi_2\ldots\pi_n$ where $\pi_k=\pi(k)$ for any $k\in[n]$. We say a permutation $\pi\in\S_n$ is a \emph{cyclic} permutation if it is composed of a single $n$-cycle. In this paper, we denote by $C(\pi)$ the cycle notation of a cyclic permutation $\pi$, starting with the element 1; that is, $C(\pi) = (1,c_2,c_3,\ldots,c_n)$ where $c_{i+1} = \pi(c_i).$ For example, $C(41532) = (1,4,3,5,2)$.

Pattern avoidance is typically defined in terms of a permutation's one-line notation. We say that a permutation $\pi\in\S_n$ \emph{avoids} a pattern $\sigma\in\S_k$ if there is no subsequence of $\pi$ that appears in the same relative order as $\sigma.$ For example, the permutation $\pi=32185476$ avoids the pattern 231, but it does contain the pattern 123 since the subsequence $\pi_2\pi_5\pi_7=257$ is in the same relative order as 123.  Enumeration of pattern-avoiding permutations began in earnest in 1985 with \cite{SS85} and has been a topic of study since that time.

Pattern avoidance of the cycle form of a cyclic permutation has been considered previously in \cite{ABBGJ}, in which the authors consider cyclic permutations where the one-line notation of $\pi$ avoids one pattern in $\S_3$ and the cycle form $C(\pi)$ avoids another pattern in $\S_3.$ It has also been considered in \cite{AL24} in which both the one-line form of a permutation and its standard cycle form (i.e.~its image under the so-called fundamental bijection) avoids a pattern in $\S_3.$ Interestingly, the  avoidance of patterns in the cycle form of a permutation has recently been used in \cite{BT23} to characterize so-called shallow permutations, which are equivalent to unlinked permutations \cite{W22}.
Avoidance of the cycle form has also been considered when characterizing almost-increasing permutations in \cite{AL17,E11} and when counting occurrences of adjacent $q$-cycles in permutations in \cite{CU24}.
Another related topic is the study of circular pattern avoidance \cite{Callan,V03,DELMSSS}, in which all cyclic rotations of a permutation avoid a given pattern.  This topic becomes particular relevant in Section~\ref{sec:allcycles} of this paper.


In this paper, we consider cyclic permutations that avoid the decreasing pattern $\delta_k=k(k-1)\ldots 21$ in its one-line notation and another pattern in its cycle form $C(\pi)$ or in all cyclic rotations of $C(\pi).$ In Section~\ref{sec:standard}, we let $\A_{n}(\sigma; \tau)$ denote the set of permutations $\pi\in\S_n$ that so that the one-line notation avoids $\sigma$ and the cycle form $C(\pi)$ avoids $\tau$; we also denote $a_{n,k}(\tau) = |\A_{n}(\delta_k; \tau)|$. For example, $\pi = 5 3 1 2 6 9 4 7 8$ avoids 4321 with $C(\pi)= (1, 5, 6, 9, 8, 7, 4, 2, 3)$ avoiding 213, so $\pi\in\A_9(4321;213).$ In this paper, we find $\A_n(\delta_k;\tau)$ for $\tau\in\S_3\setminus\{321\}.$

In particular, in Theorem~\ref{thm:213-k21} in Section~\ref{sec:213}, we find that for $k\geq 4$, $f_k(z;213) = \sum_{n\geq0}a_{n,k}(213)z^n$ is given by the height $k$ continued fraction
\[
\setlength{\arraycolsep}{0pt}
\newcommand{\md}[2][1.45]{\mathbin{\raisebox{-#1ex}[0pt][0pt]{$\displaystyle#2$}}}
\md{f_{k}(z;213)={}}
\begin{array}[t]{ *{9}{>{\displaystyle{\mathstrut}}c<{{}}} }
	z \\
	\cline{1-3}
	\md{1} & \md{-}   &z \\
	\cline{3-5}
	&          & \md{1} & \md{-} & z \\
	\cline{5-7}
	&          &          &        & \md{1} & \md{-} & \md[3]{\ddots} \\
	&&&&&&& \md{-} & z \\
	\cline{9-9}
	&&&&&&&   & 1-z.
\end{array}
\]
\sloppypar{In Section~\ref{sec:312}, we show $f_{k}(z;312)=f_k(z;213).$ We also find in Theorem~\ref{thm:231} of Section~\ref{sec:231} that if $f_k(z;231)=\sum_{n\geq 0}a_{n,k}(231)z^n,$ we have a generating function that depends on the continued fraction for $f_k(z;213)$:}
\[f_k(z;231) = \frac{f_{k-1}(z;213)\left([f_{k-2}(z;213)]^2-f_{k-2}(z;213)+1\right)}{1-z}. \]

In Section~\ref{sec:allcycles}, we let $\A^\circ_{n}(\sigma; \tau)$ 
denote the set of permutations $\pi\in\S_n$ so that the one-line notation avoids $\sigma$ and \emph{all cyclic rotations} of the cycle form $C(\pi)$ avoid $\tau$; we denote
$a^{\circ}_{n,k}(\tau) = |\A^\circ_{n}(\delta_k; \tau)|$.
For example, $\pi =3 9 2 1 4 5 6 7 8$ avoids 4321 with all cyclic rotations of $C(\pi)= (1, 3, 2, 9, 8, 7, 6, 5, 4)$ avoiding 1342, so $\pi\in\A^\circ_9(4321;1342).$ Up to symmetry and cyclic rotation, there are three cases in $\S_4$ to consider, and we find the enumeration for two of these cases. In Theorem~\ref{thm:1324}, we find that 
$a_{n,3}^\circ(1324)=2^{n-2}$ and $a_{n,k}^\circ(1324)=F_{2n-3}$ for $k\geq 4,$ where $F_n$ is the $n$th Fibonacci number. In Theorem~\ref{thm:1342} we find that if $f^\circ_k(z;1342) = \sum_{n\geq 0}a^\circ_{n,k}(1342)z^n,$ then
\[f^\circ_k(z;1342) = \frac{1-3z+2z^2+z^3}{(1-z)^2(1-2z)} - \frac{2z^{k+1}}{(1-z)^{k-1}(1-2z)}.\]
We conclude the paper in Section~\ref{sec:open} with some conjectures, open questions, and directions for further research.

\section{Avoidance in the cycle form $C(\pi)$}\label{sec:standard}

In this section, we enumerate cyclic permutations so that $\pi$ avoids the pattern $\delta_k=k(k-1)\ldots 21$ and the cycle form $C(\pi)$ (i.e., the cycle form that begins with 1) avoids the pattern $\tau$ for some $\tau\in\S_3$. We solve this for all but one $\tau \in S_3$; the case when $\tau=321$ is left as an open question.

We first note that the case when $\tau=123$ or $\tau=132$ is trivial. Indeed, the only permutation $\pi\in\S_n$ so that $C(\pi)$ avoids 123 is $\pi=n12\ldots(n-1)$ with $C(\pi)=(1,n,n-1,\ldots,3,2)$ and so $a_{n,k}(123)=1$ for all $n\geq 1$ and $k\geq 3.$ Similarly, the only permutation $\pi\in\S_n$ so that $C(\pi)$ avoids 132 is $\pi=23\ldots n1$ with $C(\pi)=(1,2,3,\ldots,n)$ and so $a_{n,k}(132)=1$ for all $n\geq 1$ and $k\geq 3.$

We will next consider the case when $\tau=213$, followed by the case when $\tau=231$, both of which give interesting answers in terms of generating functions. 
\subsection{Enumerating $\A_{n}(\delta_k;213)$}\label{sec:213}

In this section, we consider cyclic permutations that avoid $\delta_k:=k(k-1)\cdots 21$ in one-line notation and avoid $213$ in cycle form.  Notice that if $\pi$ is a permutation avoiding $\delta_k$, then $\pi$ necessarily avoids $\delta_i$ for all $i < k$. Thus,  $a_{n,k}(213) \geq a_{n,k-1}(213)$ for $k \geq 2$. Furthermore, since cyclic permutations on $[n]$ avoiding 213 in cycle form are in one-to-one correspondence with permutations on $[n-1]$ avoiding 213, we have $a_{n,k}(213)$ is the $(n-1)$-st Catalan number for $k \geq n$.

As an example, consider $n=5$, and notice there are 14 cyclic permutations on 5 elements whose cycle form avoids 213.  All 14 of these permutations avoid $\delta_5=54321$ in one-line notation, and all but the permutation $54132$ avoid $\delta_4=4321$. There are eight cyclic permutations that avoid 213 in cycle form while avoiding $\delta_3=321$ in one-line notation, namely, 31452, 31524, 41253, 51234, 23451, 23514, 24153, and 25134.

Thus we have
\[
a_{5,k}(213) = \begin{cases} 0 & \text{if } k \leq 2\\
	8 & \text{if } k = 3\\
	13 & \text{if } k = 4\\
	14 & \text{if } k \geq 5. \end{cases} \]
Figure~\ref{fig:table213} shows values of $a_{n,k}(213)$ for other small values of $n$ and $k$.

\begin{figure}\label{fig:table213}
	\begin{center}
		\begin{tabular}{c||*{9}{p{6.5mm}|}}
			\backslashbox{$k$}{$n$} & \multicolumn{1}{c|}{\textbf{1}} &  \multicolumn{1}{c|}{\textbf{2}}&  \multicolumn{1}{c|}{\textbf{3}} &  \multicolumn{1}{c|}{\textbf{4}}&  \multicolumn{1}{c|}{\textbf{5}} &  \multicolumn{1}{c|}{\textbf{6}} & \multicolumn{1}{c|}{\textbf{7}} &  \multicolumn{1}{c|}{\textbf{8}} &  \multicolumn{1}{c|}{\textbf{9}}\\ \hline \hline
			\textbf{2} & 1 & 0 & 0 & 0 & 0 & 0 & 0 & 0 & 0\\
			\textbf{3} & 1 & 1 & 2 & 4 & 8 & 16 & 32 & 64 & 128\\
			\bf 4 & 1 & 1 & 2 & 5 & 13 & 34 & 89 & 233 & 610\\
			\bf 5 & 1 & 1 & 2 & 5 & 14 & 41 & 122 & 365 & 1094\\
			\bf 6 & 1 & 1 & 2 & 5 & 14 & 42 & 131 & 417 & 1341\\
			\bf 7 & 1 & 1 & 2 & 5 & 14 & 42 & 132 & 428 & 1416\\
			\bf 8 & 1 & 1 & 2 & 5 & 14 & 42 & 132 & 429 & 1429\\
			\bf 9 & 1 & 1 & 2 & 5 & 14 & 42 & 132 & 429 & 1430\\
		\end{tabular} 
		\caption{The number of permutations in $\A_n(\delta_k;213)$; that is, the number of cyclic permutations of length $n$ avoiding $k(k-1)\cdots1$ in one-line notation and avoiding $213$ in cycle form.}
	\end{center}
\end{figure}

In general, we will show that the permutations in  $\A_n(\delta_k;213)$ can be enumerated using a self-convolution recurrence as stated in the following theorem. The corresponding generating function can be found in Corollary~\ref{cor:213-k21}.

\begin{theorem}\label{thm:213-k21}
	For $n\geq 2$ and $k\geq 4$, \[a_{n,k}(213) = \sum_{i=1}^{n-1} a_{i,k}(213)a_{n-i,k-1}(213),\]
	with base cases
	\begin{itemize}
		\item $a_{1,1}(213)=0$ and $a_{1,k}(213) = 1$ for $k \geq 2$;
		\item $a_{n,1}(213) = a_{n,2}(213) = 0$ for $n \geq 2$; and
		\item $a_{n,3}(213) = 2^{n-2}$ for $n \geq 2$.
	\end{itemize}
\end{theorem}

The remainder of this section will be devoted to proving this theorem. We begin by introducing new notation.  Let $\B_n(\delta_k;213)$ be the set of permutations in $\A_n(\delta_k;213)$ with the additional property that $\pi_1=n$, and let $b_{n,k}(213) = |\B_n(\delta_k;213)|$. It turns out that each permutation in $\pi \in \A_n(\delta_k;213)$ can be decomposed into two unique permutations based on the value of $\pi_1$; if $\pi_1 = j$, one of the permutations in the decomposition is in $\B_j(\delta_k;213)$, while the other is in $\A_{n+1-j}(\delta_k;213)$. 


\begin{lemma}\label{lem:as-k21} 
	For $n\geq 2$ and $k\geq 3$, \[a_{n,k}(213) = \sum_{j=2}^{n} b_{j,k}(213)a_{n+1-j,k}(213).\]
\end{lemma}
\begin{proof} Let $j \in [2,n]$. We show that the number of permutations in $\A_n(\delta_k;213)$ where $\pi_1 = j$ is equal to $ b_{j,k}(213)a_{n+1-j,k}(213)$. To this end, let $\pi \in \A_n(\delta_k;213)$ and suppose that $\pi_1=j$. Since $C(\pi)$ avoids 213, we can write \[
	C(\pi) = (1,j, c_3, c_4, \ldots, c_{n-j+2},c_{n-j+3},\ldots, c_{n})
	\]
	where $\{c_3,\ldots, c_{n-j+2}\} = [j+1,n]$ and $\{c_{n-j+3},\ldots, c_{n}\} = [2,j-1]$. In one-line form, we then have  \[\pi = j\pi_2\pi_3\ldots\pi_{j-1}\pi_j\ldots \pi_n\] where $\{\pi_2,\ldots,\pi_{j-1}\} = [1,j-1]\setminus\{c_{n-j+3}\}$ and $\{\pi_j,\ldots,\pi_{n}\} = [j+1,n]\cup\{c_{n-j+3}\}$. Notice that this implies that if $\pi$ did contain $\delta_k$ as a pattern, then the $\delta_k$ pattern must  be composed of only elements of $[1,j]$ or only elements of $[j+1,n]\cup\{c_{n-j+3}\}$.
	
	We now create two cyclic permutations from $\pi$ that both avoid $\delta_k$.  Define $\pi'$ to be the permutation formed by deleting the elements in $[2,j]$ from $C(\pi)$. More formally, \[ C(\pi') = (1,c_3-j+1, c_4-j+1, \ldots, c_{n-j+2}-j+1). \] Equivalently, in one-line notation, $\pi'$ can be thought of as deleting the elements in $\{\pi_1, \pi_2, \cdots, \pi_{j-1}\}$ from $\pi$. Thus $\pi'$ consists of the last $n-j+1$ elements of $\pi$ in the same relative order.
	Therefore $\pi'$ is cyclic with $C(\pi')$ avoiding $213$ and $\pi'$ avoids $\delta_k$ in one-line notation, and we have $\pi' \in \A_{n+1-j}(\delta_k; 213)$.
	
	Similarly, define $\pi''$ to be the permutation formed by deleting the elements in $[j+1,n]$ from $C(\pi)$. In this case, \[C(\pi'') = (1, j, c_{n-j+3},\ldots, c_{n}), \] and
	\[ \pi'' = j\pi_2\pi_3\cdots\pi_{j-1}c_{n-j+3}. \] We again see that $\pi''$ is cyclic, its cycle form avoids 213, and its one-line form avoids $\delta_k$. The permutation $\pi''$ has the additional condition that 1 maps to $j$ and thus $\pi'' \in \B_{j}(\delta_k; 213)$.
	
	Conversely, suppose $\pi'$ is any permutation in $\A_{n+1-j}(\delta_k; 213)$ and $\pi''$ is any permutation in $\B_{j}(\delta_k; 213)$ with
	\[ C(\pi') = (1, c_2', c_3',\ldots, c_{n+1-j}') \]
	and
	\[ C(\pi'') = (1, j, c_3'', c_4'', \ldots, c_j'').\]
	Create the permutation $\pi$ by
	\[ C(\pi) = (1, j, c_2'+ j-1, c_3' + j-1, \ldots, c_{n+1-j}' + j-1, c_3'', c_4'', \ldots, c_j'').\]
	Notice $C(\pi)$ avoids 213 and its one-line notation must avoid $\delta_k$ since the one-line notation of both $\pi'$ and $\pi''$ did.  Thus the process is invertible and the number of permutations in $\A_n(\delta_k;213)$ where $\pi_1 = j$ is equal to $ b_{j,k}(213)a_{n+1-j,k}(213)$. Summing over all possible $j$ gives the desired result.
\end{proof}

The following example, along with Figure~\ref{fig:213-k21-1}, illustrates the bijective correspondance between $\A_n(\delta_k;213)$ and $\B_j(\delta_k;213) \times \A_{n+1-j}(\delta_k;213)$.
\begin{example}
	Consider the permutation $\pi'=4 315672=(1, 4, 5, 6, 7, 2, 3) \in \A_7(4321;213)$ and the permutation $\pi''=5 3 4 1 2=(1, 5, 2, 3, 4) \in \B_5(4321;213)$. Then we can obtain a permutation $\pi \in \A_{11}(4321;213)$ as described in the preceding proof: 
	\[
	\pi = 5\ 3\ 4\ 1\ 8\ 7\ 2\ 9\ 10\ 11\ 6=(1, 5, 8, 9, 10, 11, 6, 7, 2, 3, 4),
	\]
	as shown in Figure~\ref{fig:213-k21-1}.
\end{example}
\begin{figure}[htp]
	\centering
	\resizebox{\textwidth}{!}{
		\begin{tabular}{c}
			$\pi=5\ 3\ 4\ 1\ 8\ 7\ 2\ 9\ 10\ 11\ 6$\\
			$=(1,5,8,9,10,11,6,7,2,3,4)$ \\
			\begin{tikzpicture}[scale=.5]
				\newcommand\myx[1][(0,0)]{\pic at #1 {myx};}
				\tikzset{myx/.pic = {\draw [ultra thick]
						(-2.5mm,-2.5mm) -- (2.5mm,2.5mm)
						(-2.5mm,2.5mm) -- (2.5mm,-2.5mm);}}
				\draw[gray] (0,0) grid (11,11);
				\foreach \x/\y in {
					1/5,
					2/3,
					3/4,
					4/1,
					5/8,
					6/7,
					7/2,
					8/9,
					9/10,
					10/11,
					11/6
				}
				\myx[(\x-.5,\y-.5)];
				\draw[-, ultra thick,red!70!black] (.5,.5)--(.5,4.5)
				--(4.5,4.5);
				\draw[-, ultra thick,red!70!black] (4.5,1.5)
				--(1.5,1.5)--(1.5,2.5)
				--(2.5,2.5)--(2.5,3.5)
				--(3.5, 3.5)--(3.5,.5)--(.5,.5);
				\draw[-, ultra thick,green!70!black] (4.5,4.5)--(4.5,7.5)
				--(7.5,7.5)--(7.5,8.5)
				--(8.5,8.5)--(8.5,9.5)
				--(9.5,9.5)--(9.5,10.5)
				--(10.5, 10.5)--(10.5,5.5)
				--(5.5,5.5)--(5.5,6.5)
				--(6.5,6.5)--(6.5,1.5)
				--(4.5,1.5);
			\end{tikzpicture}
		\end{tabular}
		\begin{tabular}{ccc}
			&$\pi' = 4 \ 3 \ 1 \ 5 \ 6 \ 7 \ 2$ & $\pi''=5 \ 3 \  4 \ 1 \ 2$ \\ 
			&$=(1,4,5,6,7,2,3)$ & $=(1,5,2,3,4)$ \\
			$\longleftrightarrow$& \begin{tabular}{c}
				\begin{tikzpicture}[scale=.5]
					\newcommand\myx[1][(0,0)]{\pic at #1 {myx};}
					\tikzset{myx/.pic = {\draw [ultra thick]
							(-2.5mm,-2.5mm) -- (2.5mm,2.5mm)
							(-2.5mm,2.5mm) -- (2.5mm,-2.5mm);}}
					\draw[gray] (0,0) grid (7,7);
					\foreach \x/\y in {
						1/4,
						2/3,
						3/1,
						4/5,
						5/6,
						6/7,
						7/2
					}
					\myx[(\x-.5,\y-.5)];
					\draw[-, ultra thick,green!70!black] (.5,.5)--(.5,3.5)
					--(3.5,3.5)--(3.5,4.5)
					--(4.5,4.5)--(4.5,5.5)
					--(5.5,5.5)--(5.5,6.5)
					--(6.5,6.5)--(6.5,1.5)
					--(1.5,1.5)--(1.5,2.5)
					--(2.5,2.5)--(2.5,.5)
					--(.5,.5);
			\end{tikzpicture} \end{tabular}&
			\begin{tabular}{c}
				\begin{tikzpicture}[scale=.5]
					\newcommand\myx[1][(0,0)]{\pic at #1 {myx};}
					\tikzset{myx/.pic = {\draw [ultra thick]
							(-2.5mm,-2.5mm) -- (2.5mm,2.5mm)
							(-2.5mm,2.5mm) -- (2.5mm,-2.5mm);}}
					\draw[gray] (0,0) grid (5,5);
					\foreach \x/\y in {
						1/5,
						2/3,
						3/4,
						4/1,
						5/2
					}
					\myx[(\x-.5,\y-.5)];
					\draw[-, ultra thick,red!70!black] (.5,.5)--(.5,4.5)
					--(4.5,4.5)--(4.5,1.5)
					--(1.5,1.5)--(1.5,2.5)
					--(2.5,2.5)--(2.5,3.5)
					--(3.5, 3.5)--(3.5,.5)--(.5,.5);
				\end{tikzpicture}
			\end{tabular}
		\end{tabular}
	}
	\caption{The permutation $\pi\in\A_{11}(4321;213)$ with $\pi_{1}=5$ can be decomposed into a permutation $\pi'~\in~\A_7(4321;213)$ and $\pi''\in\B_5(4321;213)$.}
	\label{fig:213-k21-1}
\end{figure}
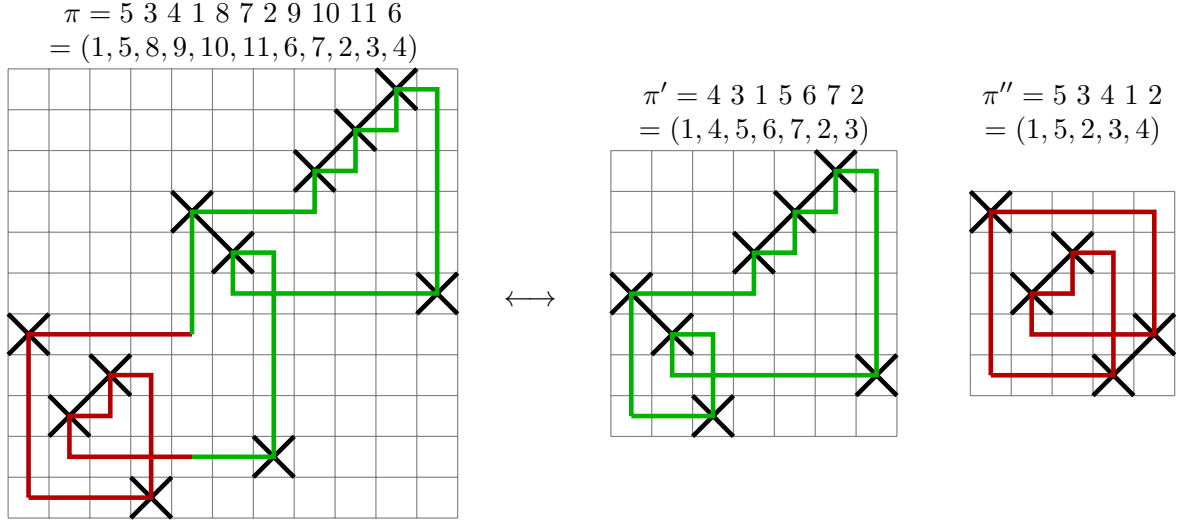

Because each element in $\A_n(\delta_k;213)$ can be decomposed as two permutations in $\B_j(\delta_k; 213)$ and in $\A_{n+1-j}(\delta_k;213)$, being able to enumerate $\B_j(\delta_k;213)$ will help us enumerate $\A_n(\delta_k;213)$. We will enumerate $\B_n(\delta_k;213)$ by first counting all of those permutations in $\B_n(\delta_k;213)$ with the additional condition  $\pi_n=2$. The next lemma shows that set of permutations in $\B_n(\delta_k;213)$ with  $\pi_n=2$ is in bijective correspondence with $\A_{n-2}(\delta_{k-2};213)$.
\begin{lemma}\label{lem:n2s-k21}
	For $n\geq 3$, the number of permutations in $\B_n(\delta_k;213)$ with $\pi_n=2$ is equal to $a_{n-2,k-2}(213)$. 
\end{lemma}
\begin{proof}
	Suppose $\pi$ is a cyclic permutation on $n$ elements where $C(\pi)$ avoids 213, and with the additional conditions that $\pi_1 = n$ and $\pi_n=2$. Then \[C(\pi) = (1,n,2,c_4,\ldots, c_n).\]
	Now consider the permutation $\pi'$ where $C(\pi')$ is formed by deleting $n$ and $2$ from $C(\pi)$. In other words,
	\[C(\pi') = (1,c_4-1,\ldots,c_n-1).\] 
	In one-line notation, we note that $\pi'$ is formed from $\pi$ by deleting $n$ and $2$, and conversely, $\pi$ is obtained from $\pi'$ by inserting $n$ at the front and $2$ at the end of $\pi'$.  
	
	We further claim that if $\delta_{k-2}$ is a pattern in $\pi'$, then there must be an occurrence of $\delta_{k-2}$ that does not use the element 1. Toward a contradiction, suppose we have an occurrence of $\delta_{k-2}$ in $\pi'$ that uses 1, namely $\pi'_{i_1}\ldots \pi'_{i_{k-2}}$ with $\pi'_{i_{k-2}}=1$. Then, we can show that $2$ must occur after 1 in $\pi'$ and thus the pattern induced by $\pi'_{i_1}\ldots \pi'_{i_{k-3}}2$ is also a $\delta_{k-2}$ pattern. To see this suppose 2 appears before 1 in $\pi'$. Then if $\pi_r=1$ and $\pi_s=2$ with $1<s<r$, the pattern $s2r$ appears in $C(\pi')$, but this is a 213 pattern, contradicting that $C(\pi')$ avoids 213. Notice that if $s=1,$ then the cycle form starts $(1,2,\ldots)$, and thus we can delete the 2 and apply the same argument since this 2 cannot be part of a 213 pattern in $C(\pi)$ or a $\delta_k$ in the one-line notation of $\pi$. This implies that $\pi'$ avoids $\delta_{k-2}$ if and only if $\pi$ avoids $\delta_k$.

	We have shown that $\pi \in \B_n(\delta_k;213)$ with $\pi_n=2$ if and only if $\pi' \in \A_{n-2}(\delta_{k-2};213)$ and therefore the result holds.
\end{proof}


To count those permutations in $\B_n(\delta_k;213)$ where $\pi_n =j \in [3,n-1]$, we can consider a decomposition based on $j$ as seen in the next result.

\begin{lemma}\label{lem:bs-k21} 
	For $n\geq 3$ and $k\geq 2$, \[b_{n,k} (213) = \sum_{j=2}^{n-1} b_{j,k}(213)a_{n-j,k-2}(213).\] 
\end{lemma}
\begin{proof}
	We show that the number of permutations in $\B_n(\delta_k;213)$ where $\pi_n = j \in [3,n-1]$ is equal to $ b_{j,k}(213)a_{n+1-j,k}(213)$. To this end, let $\pi \in \B_n(\delta_k;213)$ and suppose that $\pi_n=j$. 
	
	Since $C(\pi)$ avoids 213, we can write \[
	C(\pi) = (1,n,j, c_4, c_5, \ldots, c_{n-j+2},c_{n-j+3},\ldots, c_{n})
	\]
	where $\{c_4,\ldots, c_{n-j+2}\} = [j+1,n-1]$ and $\{c_{n-j+3},\ldots, c_{n}\} = [2,j-1]$. This implies that $\pi$ must be of the form \[\pi = n\pi_2\pi_3\ldots\pi_{j-1}\pi_j\ldots \pi_{n-1}j\] where $\{\pi_2,\ldots,\pi_{j-1}\} = [1,j-1]\setminus\{c_{n-j+3}\}$ and $\{\pi_j,\ldots,\pi_{n-1}\} = [j+1,n-1]\cup\{c_{n-j+3}\}$. Now, if $\pi$ contains $\delta_k$ as a pattern, it must either contain only elements of $[1,j]\cup\{n\}$ or only elements of $[j,n]\cup\{c_n-j+2\}$. Therefore, if $\pi$ avoids $\delta_k$, we can  decompose $\pi$ into two smaller cyclic permutations, each of which avoids $\delta_k$, namely, $\pi'\in \B_{n-j+2}(\delta_k;213)$ and $\pi''\in\B_{j}(\delta_k;213)$, defined by 
	\[C(\pi') = (1, n-j+2, 2, c_4-j+2, c_5-j+2, \ldots, c_{n-j+2}-j+2) \]
	and 
	\[C(\pi'') = (1, j, c_{n-j+3},\ldots, c_{n}).\]
	Notice that $\pi'$ can be any permutation in $\B_{n-j+2}(\delta_k;213)$ with $\pi'_1=n-j+2$ (the largest element) and $\pi'_{n-j+2}=2$, while $\pi''$ can be any permutation in $\B_{j}(\delta_k;213)$ with $\pi''_1=j$. By Lemma~\ref{lem:n2s-k21}, the number of such $\pi'$ is given by $a_{n-j,k-2}(213)$ and so the recurrence follows. 
\end{proof}

We again illustrate this decomposition with a specific example.

\begin{example} Consider the permutation $$\pi= 13\ 1 \ 2\ 3\ 7\ 4\ 6 \ 9\ 12\ 5\ 10\ 11\ 8
	=(1,13,8,9,12,11,10,5,7,6,4,3,2).$$ Notice that $\pi \in \B_{13, 4}$ and that $\pi_{13} = 8.$ Following the proof of the preceding lemma, we define two new permutations, $\pi'$ and $\pi''$ by
	\[ C(\pi') = (1, 7,2,3,6,5,4) \quad \text{and} \quad C(\pi'') = (1,8,5,7,6,4,3,2).\]
	We note that $\pi' \in \B_7(\delta_4;213)$ and $\pi'' \in \B_8(\delta_4;213)$ as desired. A pictorial representation of this decomposition is shown in Figure~\ref{fig:213-k21-2}.
\end{example}
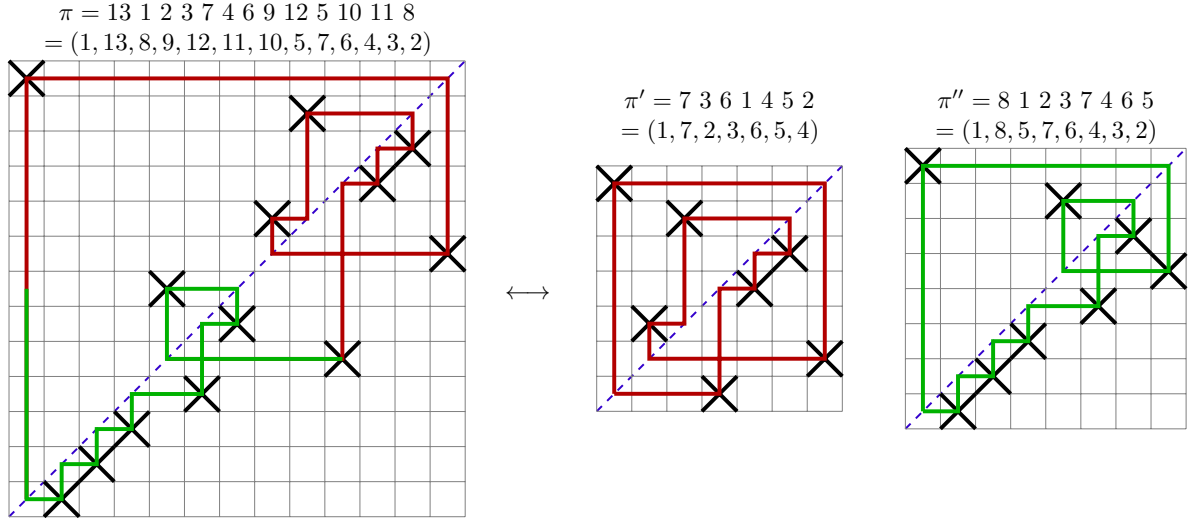
\begin{figure}[htp]
	\centering
	\resizebox{\textwidth}{!}{
		\begin{tabular}{c}
			$\pi= 13\ 1 \ 2\ 3\ 7\ 4\ 6 \ 9\ 12\ 5\ 10\ 11\ 8$ \\
			$=(1,13,8,9,12,11,10,5,7,6,4,3,2)$\\
			\begin{tikzpicture}[scale=.5]
				\newcommand\myx[1][(0,0)]{\pic at #1 {myx};}
				\tikzset{myx/.pic = {\draw  [ultra thick]
						(-2.5mm,-2.5mm) -- (2.5mm,2.5mm)
						(-2.5mm,2.5mm) -- (2.5mm,-2.5mm);}}
				\draw[blue!80!red, dashed, thick] (0,0)--(13,13);
				\draw[gray] (0,0) grid (13,13);
				\foreach \x/\y in {
					1/13,
					2/1,
					3/2,
					4/3,
					5/7,
					6/4,
					7/6,
					8/9,
					9/12,
					10/5,
					11/10,
					12/11,
					13/8
				}
				\myx[(\x-.5,\y-.5)];
				\draw[-, ultra thick,red!70!black] (.5,.5)-- (.5,12.5)
				--(12.5,12.5)--(12.5,7.50000000000000)
				--(7.50000000000000,7.50000000000000)--(7.5,8.5)--(8.50000000000000,8.50000000000000)--(8.50000000000000,11.50000000000000)
				--(11.50000000000000,11.50000000000000)--(11.50000000000000,10.50000000000000)-- (10.50000000000000,10.50000000000000)--
				(10.50000000000000,9.50000000000000)
				--(9.50000000000000,9.50000000000000)--(9.50000000000000,4.50000000000000)
				--(4.50000000000000,4.50000000000000);
				\draw[-, ultra thick,green!70!black] (9.50000000000000,4.50000000000000)
				--(4.50000000000000,4.50000000000000)--(4.500000000000000,6.5)--(6.5,6.5)--(6.5,5.5)
				--(5.5,5.5)--(5.5,3.5)
				--(3.5,3.5)--(3.5,2.5)
				--(2.5,2.5)--(2.5,1.5)
				--(1.5,1.5)--(1.5,.5)
				--(.5,.5)--(.5,6.5);
			\end{tikzpicture} 
		\end{tabular}
		\begin{tabular}{ccc}
			&$\pi' = 7 \ 3\ 6\ 1\ 4 \ 5 \ 2$ & $\pi'' = 8\ 1\ 2\ 3\ 7\ 4\ 6\ 5$  \\ 
			&$=(1,7,2,3,6,5,4)$ & $=(1,8,5,7,6,4,3,2)$ \\
			$\longleftrightarrow$& \begin{tabular}{c}
				\begin{tikzpicture}[scale=.5]
					\newcommand\myx[1][(0,0)]{\pic at #1 {myx};}
					\tikzset{myx/.pic = {\draw [ultra thick]
							(-2.5mm,-2.5mm) -- (2.5mm,2.5mm)
							(-2.5mm,2.5mm) -- (2.5mm,-2.5mm);}}
					\draw[blue!80!red, dashed, thick] (0,0)--(7,7);
					\draw[gray] (0,0) grid (7,7);
					\foreach \x/\y in {
						1/7,
						2/3,
						3/6,
						4/1,
						5/4,
						6/5,
						7/2
					}
					\myx[(\x-.5,\y-.5)];
					\draw[-, ultra thick,red!70!black] (.5,.5)--(.5,6.5)
					--(6.5,6.5)--(6.5,1.5)
					--(1.5,1.5)--(1.5,2.5)
					--(2.5,2.5)--(2.5,5.5)
					--(5.5,5.5)--(5.5,4.5)
					--(4.5,4.5)--(4.5,3.5)
					--(3.5,3.5) -- (3.5,.5)
					--(0.5,.5);
				\end{tikzpicture}
			\end{tabular}&
			\begin{tabular}{c}
				\begin{tikzpicture}[scale=.5]
					\newcommand\myx[1][(0,0)]{\pic at #1 {myx};}
					\tikzset{myx/.pic = {\draw  [ultra thick]
							(-2.5mm,-2.5mm) -- (2.5mm,2.5mm)
							(-2.5mm,2.5mm) -- (2.5mm,-2.5mm);}}
					\draw[blue!80!red, dashed, thick] (0,0)--(8,8);
					\draw[gray] (0,0) grid (8,8);
					\foreach \x/\y in {
						1/8,
						2/1,
						3/2,
						4/3,
						5/7,
						6/4,
						7/6,
						8/5
					}
					\myx[(\x-.5,\y-.5)];
					\draw[-, ultra thick,green!70!black] (.5,.5)--(.5,7.5)
					--(7.5,7.5)--(7.5,4.50000000000000)
					--(4.50000000000000,4.50000000000000)--(4.50000000000000,6.50000000000000)
					--(6.50000000000000,6.50000000000000)--(6.50000000000000,5.50000000000000)
					--(5.50000000000000,5.50000000000000)--(5.50000000000000,3.50000000000000)
					--(3.50000000000000,3.50000000000000)--(3.50000000000000,2.50000000000000)
					--(2.50000000000000,2.50000000000000)--(2.50000000000000,1.50000000000000)
					--(1.50000000000000,1.50000000000000)--(1.50000000000000,0.500000000000000)
					--(0.500000000000000,.5);
			\end{tikzpicture} \end{tabular}
		\end{tabular}
	}
	\caption{The permutation $\pi\in\B_{13}(4321;213)$ with $\pi_{13}=8$ can be decomposed into a permutation $\pi'~\in~\B_7(4321;213)$ and $\pi''\in\B_8(4321;213)$.}
	\label{fig:213-k21-2}
\end{figure}

Lemma~\ref{lem:bs-k21} gives $b_{n,k}$ as a sum of products of $b_{j,k}$ and $a_{n-j,k-2}$ (for $j$ from 2 to $n-1$) while Lemma~\ref{lem:as-k21} gives $a_{n,k}$ as a sum of products of $b_{j,k}$ and $a_{n+1-j,k}$ (for $j$ from 2 to $n$). We can combine these results to get $b_{n+1,k} = a_{n,k-1}$ as seen in the next lemma.

\begin{lemma}\label{lem:k21-a=b}
	For $n\geq 3$ and $k\geq 3$, \[b_{n+1,k}(213) = a_{n,k-1}(213).\]
\end{lemma}
\begin{proof}
	We prove this by induction. It is straightforward to check the base cases. Let us assume the statement is true for all $N<n$ and $K<k$. Notice that by Lemmas \ref{lem:as-k21} and \ref{lem:bs-k21}, we can write 
	\[b_{n+1,k}(213) = \sum_{i=2}^{n} b_{i,k}(213)a_{n+1-i,k-2}(213)
	\]
	and
	\[a_{n,k-1}(213) = \sum_{j=2}^{n} b_{j,k-1}(213)a_{n+1-j,k-1}(213).\]
	Using induction, we can say that $b_{j,k-1}(213) = a_{j-1,k-2}(213)$ and $a_{n+1-j,k-1}(213)=b_{n+2-j,k}(213)$. Reindexing, we can see that the new formula for $a_{n,k-1}(213)$ coincides with that for $b_{n+1,k}(213)$.
\end{proof}

We are now ready to prove our main result as stated in Theorem~\ref{thm:213-k21}.

\begin{proof}[Proof of Theorem~\ref{thm:213-k21}]
	The base cases $a_{1,1}(213)=0$ and $a_{1,k}(213)=1$ for $k\geq 2$ are clear since there is only one permutation of length $1$, which must avoid any pattern of length $k\geq 2.$ For $n\geq2$, it is also clear that $a_{n,1}(213)=a_{n,2}(213)=0$ since it is impossible for any permutation in $\S_n$ to avoid the pattern of length 1 and impossible for a cyclic permutation to avoid $21$. 
	The base case $a_{n,3}(212)=2^{n-2}$ for $n\geq 2$ can be found in \cite[Theorem 3.24]{ABBGJ}.
	Finally, by combining Lemmas \ref{lem:as-k21} and \ref{lem:k21-a=b} and reindexing, the recurrence in the statement of theorem follows.
\end{proof}


For any fixed $k$, the generating function for $a_{n,k}(213)$ can be described recursively as seen in the following corollary.

\begin{corollary}\label{cor:213-k21} Let $f_k(z;213)$ be the generating function for the $a_{n,k}(213)$. Then for $k \geq 4$,
	\[ f_k(z;213) = \frac{z}{1-f_{k-1}(z;213)}\]
	with
	$f_1(z;213) = 0$, $f_2(z;213) = z$, and $f_3(z;213) =\frac{z(1-z)}{1-2z}$.
\end{corollary}


Notice that since $f_3(z;213)$ can also be written as 
\[
f_3(z;213) =\frac{z}{1-\frac{z}{1-z}},
\]
the generating function described in Corollary~\ref{cor:213-k21} is a continued fraction generating function. For example, for $k=4$, 
\[f_4(z;213) = \frac{z}{1-f_3(z;213)}= \frac{z}{1-\frac{z}{1-\frac{z}{1-z}}} = z + z^2 + 2z^3 + 5z^4 + 13z^5 + 34z^6 + 89z^7+ \cdots,\]
where the coefficient of $z^n$ gives the number of cyclic permutations $\pi$ on $[n]$ that avoid $4321$ whose  cycle form $C(\pi)$ avoids 213. Similarly, for $k=5$,
\[f_5(z;213) = \frac{z}{1-f_4(z;213)}= \frac{z}{1-\frac{z}{1-\frac{z}{1-\frac{z}{1-z}}}} =z + z^2 + 2 z^3 + 5 z^4 + 14 z^5 + 41 z^6 + 122 z^7 + \cdots,\] where the coefficient of $z^n$ gives the number of cyclic permutations $\pi$ on $[n]$  that avoid $54321$ whose  cycle form $C(\pi)$ avoids 213.

\begin{remark}
	For $k\geq 4$, these sequences given by the recurrence and generating function in Theorem~\ref{thm:213-k21} and Corollary~\ref{cor:213-k21} can be found in OEIS A080934. It follows that $a_k(\delta_k;213)$ is equal to the number of permutations of length $n-1$ that classically avoid both 231 and $k(k-1)\ldots 21$, as well as the number of rooted ordered trees on $n$ notes of height at most $k-1$. 
\end{remark}

\subsection{Enumerating $\A_n(\delta_k;312)$}\label{sec:312}

We now briefly consider $\A_n(\delta_k; 312)$. Since $312$ is the reverse of 213, $C(\pi)$ avoids 213 if and only if $C(\pi)^r = (c_n, c_{n-1}, \ldots, c_2, 1)$ avoids 312. Furthermore, we must have $C(\pi^{-1})  = (1, c_n, c_{n-1}, \ldots, c_2),$ which also necessarily avoids 312 since the only change from $C(\pi)^r = (c_n, c_{n-1}, \ldots, c_2, 1)$ is that the 1 was moved to the front. Finally, since the inverse of the decreasing monotone pattern is itself, we have that $a_{n,k}(213) = a_{n,k}(312)$. We summarize these results in the following theorem.

\begin{theorem}
	For $n \geq 1$ and $k \geq 1$, $a_{n,k}(312) = a_{n,k}(213)$. Furthermore, if $f_k(z;312)$ is the generating function for $a_{n,k}(312)$, then $f_k(312;z) = f_k(213;z)$.
\end{theorem}

\subsection{Enumerating $\A_n(\delta_k;231)$}\label{sec:231}

In this section, we consider cyclic permutations that avoid $\delta_k$ in one-line notation and avoid 231 in cycle form. We again define $a_{n,k}(231) = |\A_n(\delta_k; 231)|$ with a summary of small values of $a_{n,k}(231)$ found in Figure~\ref{fig:table231}. Interestingly, this case will use the results in Section~\ref{sec:213}. There are multiple useful correspondences between permutations in $\A_n(\delta_k;231)$ and cyclic permutations whose cycles avoid 213. The following lemma formalizes some of these facts.

\begin{figure}\label{fig:table231}
	\begin{center}
		\begin{tabular}{c||*{9}{p{6.5mm}|}}
			\backslashbox{$k$}{$n$} & \multicolumn{1}{c|}{\textbf{1}} &  \multicolumn{1}{c|}{\textbf{2}}&  \multicolumn{1}{c|}{\textbf{3}} &  \multicolumn{1}{c|}{\textbf{4}}&  \multicolumn{1}{c|}{\textbf{5}} &  \multicolumn{1}{c|}{\textbf{6}} & \multicolumn{1}{c|}{\textbf{7}} &  \multicolumn{1}{c|}{\textbf{8}} &  \multicolumn{1}{c|}{\textbf{9}}\\ \hline \hline
			\textbf{2} & 1 & 0 & 0 & 0 & 0 & 0 & 0 & 0 & 0\\
			\textbf{3} & 1 & 1 & 2 & 3 & 4 & 5 & 6 & 7 & 8\\
			\bf 4 & 1 & 1 & 2 & 5 & 11 & 23 & 47 & 95 & 191\\
			\bf 5 & 1 & 1 & 2 & 5 & 14 & 40 & 113 & 314 & 860\\
			\bf 6 & 1 & 1 & 2 & 5 & 14 & 42 & 130 & 406 & 1267\\
			\bf 7 & 1 & 1 & 2 & 5 & 14 & 42 & 132 & 427 & 1403 \\
			\bf 8 & 1 & 1 & 2 & 5 & 14 & 42 & 132 & 429 & 1428\\
			\bf 9 & 1 & 1 & 2 & 5 & 14 & 42 & 132 & 429 & 1430\\
		\end{tabular} 
		\caption{The number of permutations in $\A_n(\delta_k;231)$; that is, the number of cyclic permutations of length $n$ avoiding $k(k-1)\cdots1$ in one-line notation and avoiding $231$ in cycle form.}
	\end{center}
\end{figure}

\begin{lemma}\label{lem:231-1} Let $n \geq 3$ and $k \geq 5$. Then
	\begin{enumerate}
		\item The number of permutations $\pi \in \A_n(\delta_k; 231)$ with $\pi_1 = n-1$ and $\pi_n = 1$ is equal to $a_{n-2, k-2}(213)$.  
		\item The number of permutations $\pi \in \A_n(\delta_k;231)$ with $\pi_1=n$ and $\pi_n=2$ is equal to $a_{n-2,k-2}(213).$
		\item The number of permutations $\pi \in \A_n(\delta_k;231)$ with $\pi_1=n$, $\pi_{n-1}=1$, and $\pi_{n} = n-2$ is $a_{n-3,k-2}(213).$
	\end{enumerate}
\end{lemma}

\begin{proof}
	We will find a one-to-one correspondence between permutations $\pi \in \A_n(\delta_k; 231)$ with $\pi_1 = n-1$ and $\pi_n = 1$ and permutations  $\pi' \in \A_n(\delta_k;213)$ with the additional constraints that $\pi'_1 = n$ and $\pi'_n=2$.  When this correspondence is found, Lemma~\ref{lem:n2s-k21} can be used to show that the number of desired permutations is then $a_{n-2,k-2}(213)$.
	
	To that end, let $\pi \in \A_n(\delta_k;231)$ with $\pi_1 = n-1$ and $\pi_n =1$. In cycle form, we have
	\[ C(\pi) = (1, n-1, c_3, c_4, \ldots, c_{n-1}, n).\]
	Consider the permutation $\pi'=\pi^{rc}$. Since $\pi_1=n-1$ and $\pi_n=1$, we have $\pi'_1 = n$ and $\pi'_n=2.$ Furthermore, the cycle notation of $\pi'$ is complemented, and thus is given by
	\[ C(\pi') = (1, n, 2, n+1-c_3, n+1-c_4, \ldots, n+1-c_{n-1}).\]
	Thus $C(\pi')$ avoids 213, and $\pi'$ avoids $\delta_k$. Because this process is invertible, the result holds.
	
	Now suppose $\pi \in \A_n(\delta_k;231)$ with $\pi_1=n$ and $\pi_n=2$. Then
	\[ C(\pi) = (1, n, 2, c_4, c_5, \ldots, c_{n})\]
	and
	\[ \pi = n\pi_2\pi_3\cdots\pi_{n-1}2.\]  
	
	Let us define $\pi''$ by complementing $C(\pi)$ and deleting $n$, so that \[C(\pi'') = (1, n-1, n+1-c_4, n+1-c_5, \ldots, n+1-c_{n}).\]  We claim that $\pi''\in\B_{n-1}(\delta_k;213)$, and in fact all such elements in $\B_{n-1}(\delta_{k-1};213)$ can be obtained this way.  First, it is clear that $C(\pi'')$ avoids 231 since the elements after $n-1$ must avoids 231. Next, let us notice that the permutation $\pi^{rc}$, obtained by complementing the cycle form of $\pi$, is given by $\pi^{rc}=(n-1)(n+1-\pi_{n-1})\ldots(n+1-\pi_2)1$ avoids $\delta_k$. The relationship between $\pi^{rc}$ and $\pi''$ can be seen as follows. If $\pi^{rc}_\ell=n,$ then $\pi''_i = \pi^{rc}_i$ for all $i\in[n-1]\setminus\{\ell\}$ and $\pi''_\ell=1$. 
	
	To see that $\pi''$ does not contain the pattern $\delta_{k-1}$, we use contradiction. If if it did have a $\delta_{k-1}$, then by the proof of Lemma~\ref{lem:n2s-k21}, there is a $\delta_{k-1}$ pattern that does not use 1. This pattern would then also appear in $\pi^{rc},$ and together with the 1 at the end of $\pi^{rc},$ we would have a $\delta_k$ pattern in $\pi^{rc},$ which is a contradiction.  Thus $\pi''$ avoids $\delta_{k-1}$ if and only if $\pi^{rc}$ avoids $\delta_k.$ This, together with the fact that $a_{n-2, k-2}(213) = b_{n-1,k-1}(213)$ by Lemma~\ref{lem:k21-a=b} implies the second bullet point in the theorem.
	
	
	
	Finally, suppose $\pi \in \A_n(\delta_k;231)$ with $\pi_1=n$, $\pi_{n-1}=1$, and $\pi_{n} = n-2$. Then
	\[ C(\pi) = (1, n,n-2,c_4, c_5, \ldots, c_{n-1}, n-1)\] and \[ \pi = n\pi_2\cdots \pi_{n-2}1(n-2).\]
	In this case, consider the permutation $\pi'''$ formed by deleting $1$, $n$, and $n-2$ from $C(\pi)$ and then taking the complement. In other words,
	\[ C(\pi''') 
	= (1, n-1-c_4, n-1-c_5, \ldots, n-1-c_{n-1}). \]
	We note that $C(\pi''')$ is 213-avoiding since $C(\pi)$ is 231-avoiding and $231^c = 213$. In one-line notation, $\pi'''$ is formed by deleting $1$, $n$, and $n-2$ from the one-line notation of $\pi$ and then taking the reverse complement.  Thus,
	\[ \pi''' = (n-2-\pi_{n-2})(n-2-\pi_{n-3})\cdots(n-2-\pi_{2}).\]
	Since $\delta_k^{rc} = \delta_k$, we see that $\pi'''$ clearly avoids $\delta_k.$ However, any maximal decreasing sequence in $\pi$ must use $n$ and either $1$ or $n-2$, so $\pi'''$ avoids $\delta_{k-2}$ and $\pi''' \in \A_{n-3}(\delta_k; 213)$. Because this process is invertible, we have the desired result.
\end{proof}

We now illustrate the correspondences described in Lemma~\ref{lem:231-1} with some examples.

\begin{example}
	First, consider the permutation \[\pi = 11 \ 4 \ 2 \ 9 \ 3 \ 5 \ 6 \ 7 \ 10 \ 12 \ 8 \ 1\quad \text{ with }\quad C(\pi)=(1, 11, 8, 7, 6, 5, 3, 2, 4, 9, 10, 12)\]  in $\A_{12}(54321; 231)$ corresponding to the first bullet of Lemma~\ref{lem:231-1}. We can obtain \[\pi'=\pi^{rc}= 11\  5\ 1\ 3\ 6\ 7\ 8\ 10\ 4\ 11\ 9\ 2 \quad\text{ with }\quad C(\pi') = (1,12,2,5,6,7,8,10,11,9,4,3)\] which is in $\A_{12}(54321;213)$ and has the property that $\pi'_1=12$ and $\pi'_{12}=2$.
	
	Now, consider the permutation 
	\[\pi =12\ 8\ 9\ 3\ 4\ 5\ 6\ 7\ 10\ 11\ 1\ 2 \quad \text{ with }\quad C(\pi) =(1, 12, 2, 8, 7, 6, 5, 4, 3, 9, 10, 11)\] in $\A_{12}(54321;231)$  corresponding to the second bullet of Lemma~\ref{lem:231-1}. We can obtain $\pi''$ by complementing $C(\pi)$ and deleting $n=12$, obtaining
	\[
	\pi'' = C(\pi) =11\ 1\ 2\ 3\ 6\ 7\ 8\ 9\ 10\ 4\ 5 =(1,11,5,6,7,8,9,10,4,3,2)
	\]
	which is an element of $\B_{11}(4321;213)$.
	
	Finally, consider the permutation \[\pi = 83527416 \quad \text{ with } \quad C(\pi)= (1,8,6,4,2,3,5,7) \in \A_8(\delta_6; 231)\] corresponding to the third bullet of Lemma~\ref{lem:231-1}. Following the proof of Lemma~\ref{lem:231-1}, we create the corresponding permutation $\pi''' \in \A_5(\delta_4;213)$ by deleting $1,$ $8$, and $6$ from $\pi$ and then taking the reverse-complement. In this case, deleting the appropriate values yields $35274$ or $24153.$ The reverse-complement is then $\pi''' = 31524 = (1,3,5,4,2)$ which is in $\A_{5}(\delta_4;213).$
\end{example}

As in Section~\ref{sec:213}, we continue by defining $\B_n(\delta_k;231)$ to be the set of permutations in $\A_{n}(\delta_k;231)$ with the additional property that $\pi_1=n$, and we let $b_{n,k}(231) = |\B_n(\delta_k;231)|$. We can show that any permutation in $\A_{n,k}(231)$ that is \emph{not} in $\B_n(\delta_k;231)$ can be decomposed into two unique permutations based on the value of $\pi_1$. If $\pi_1=j$, one of the permutations in the decomposition is in $\A_{j+1}(\delta_k;231)$ with some additional constraints while the other is in $\A_{n+1-j}(\delta_k;231).$ The subsequent lemma enumerates permutations in $\A_{n}(\delta_k;231)$ based on this decomposition. Notice that the recurrence in the next lemma also depends on the values of $a_{n,k}(213)$ found in the previous section.

\begin{lemma}\label{lem:as-k21-3} 
	For $n\geq 3$ and $k\geq 5$, \[a_{n,k}(231) = b_{n,k}(231)+ \sum_{j=2}^{n-1} a_{j-1,k-2}(213)a_{n+1-j,k}(231).\]
\end{lemma}
\begin{proof}
	Notice that since $b_{n,k}(231)$ is the number of permutations $\pi \in \A_{n}(\delta_k; 213)$ with the property that $\pi_1=n$, we need only consider the permutations in $\A_{n,k}(231)$ that do not begin with $n$.
	Let $j \in [2,n-1]$. We show that the number of permutations in $\A_n(\delta_k;231)$ where $\pi_1 = j<n$ is equal to $ a_{j-1,k-2}(213)a_{n+1-j,k}(231)$. To this end, let $\pi \in \A_n(\delta_k;231)$ and suppose that $\pi_1=j$. Since $C(\pi)$ avoids 231, we can write \[
	C(\pi) = (1,j, c_3, c_4, \ldots, c_{j},c_{j+1},\ldots, c_{n})
	\]
	where $\{c_3,\ldots, c_{j}\} = [2,j-1]$ and $\{c_{j+1},\ldots, c_{n}\} = [j+1,n]$, which is nonempty. In one-line form, we then have 
	\[\pi = j\pi_2\pi_3\ldots\pi_{j}\pi_{j+1}\ldots \pi_n\] where $\{\pi_2,\ldots,\pi_{j}\} = [2,j-1]\cup\{c_{j+1}\}$ and $\{\pi_{j+1},\ldots,\pi_{n}\} = [j+1,n]\setminus\{c_{j+1}\}\cup\{1\}$. Notice that this implies that if $\pi$ did contain $\delta_k$ as a pattern, then the $\delta_k$ pattern must  be composed of only elements of $[1,j]\cup\{c_{j+1}\}$ or only elements of $[j+1,n]\cup\{1\}$.
	
	We now create two cyclic permutations from $\pi$ that both avoid $\delta_k$.  Define $\pi'$ to be the permutation formed by deleting the elements in $[2,j]$ from $C(\pi)$. More formally, \[C(\pi') = (1,c_{j+1}-j+1,\ldots, c_{n}-j+1).\] 
	Equivalently, in one-line notation, $\pi'$ can be thought of as deleting the elements in $[2,j]=\{\pi_1, \pi_2, \cdots, \pi_{j}\}\setminus\{\pi_{c_{j}}\}$ from $\pi$. Thus $\pi'$ is cyclic with $C(\pi')$ avoiding $231$ and $\pi'$ avoids $\delta_k$ in one-line notation, and thus $\pi' \in \A_{n+1-j}(\delta_k; 231)$.

	Similarly, define $\pi''$ to be the permutation formed by deleting the elements in $[j+1,n]\setminus\{c_{j+1}\}$ from $C(\pi)$. In this case, \[C(\pi'') = (1, j, c_{3},\ldots, c_{j}, j+1) \] and
	\[ \pi'' = j\pi_2\pi_3\cdots\pi_{j-1}\pi_j1\] where we take $\pi_{c_{j}}=j+1$.  We again see that $\pi''$ is cyclic, its cycle form avoids 231, and its one-line form avoids $\delta_k$. The permutation $\pi''\in \A_{j+1}(\delta_k; 231)$ has the additional conditions that $\pi_1=j$ and $\pi_{j+1}=1$. By Lemma~\ref{lem:231-1}, the number of such permutations is equal to $a_{j-1,k-2}(213)$.
\end{proof}

\begin{example}
	Consider the permutation
	\[5\ 8\ 2\ 3\ 4\ 9\ 6\ 7\ 11\ 1\ 10= (1, 5, 4, 3, 2, 8, 7, 6, 9, 11, 10)\] in $\A_{11}(4321;231).$ Then, we can obtain the permutations $\pi'= 4523671=(1, 4, 3, 2, 5, 6, 7)$ and $\pi''=562341=(1,5,4,3,2,6)$ as described in the preceding proof and
	as shown in Figure~\ref{fig:231-k21-1}.
\end{example}

\begin{figure}[htp]
	\centering
	\resizebox{\textwidth}{!}{
		\begin{tabular}{c}
			$\pi=5\ 8\ 2\ 3\ 4\ 9\ 6\ 7\ 11\ 1\ 10$\\
			$=(1, 5, 4, 3, 2, 8, 7, 6, 9, 11, 10)$ \\
			\begin{tikzpicture}[scale=.5]
				\newcommand\myx[1][(0,0)]{\pic at #1 {myx};}
				\tikzset{myx/.pic = {\draw [ultra thick]
						(-2.5mm,-2.5mm) -- (2.5mm,2.5mm)
						(-2.5mm,2.5mm) -- (2.5mm,-2.5mm);}}
				\draw[gray] (0,0) grid (11,11);
				\foreach \x/\y in {
					1/5,
					2/8,
					3/2,
					4/3,
					5/4,
					6/9,
					7/6,
					8/7,
					9/11,
					10/1,
					11/10
				}
				\myx[(\x-.5,\y-.5)];
				\draw[-, ultra thick,red!70!black] (.5,.5)--(.5,4.5)
				--(4.5,4.5) -- (4.5,3.5)--(3.5,3.5)--(3.5,2.5)
				--(2.5,2.5)--(2.5,1.5)--(1.5,1.5)--(1.5,7.5)--(7.5,7.5);
				\draw[-, ultra thick,red!70!black] (9.5,9.5)--(9.5,.5)--(.5,.5);
				\draw[-, ultra thick,green!70!black] (1.5,7.5)--
				(7.5,7.5)--(7.5,6.5)
				--(6.5,6.5)--(6.5,5.5)
				--(5.5,5.5)--(5.5,8.5)
				--(8.5, 8.5)--(8.5,10.5)
				--(10.5,10.5)--(10.5,9.5)
				--(9.5,9.5)--(9.5,.5);
			\end{tikzpicture}
		\end{tabular}
		\begin{tabular}{ccc}
			&$\pi' = 4 \ 5 \ 2 \ 3 \ 6 \ 7 \ 1$ & $\pi''=5 \ 6 \  2 \ 3 \ 4 \ 1$ \\ 
			&$=(1,4,3,2,5,6,7)$ & $=(1,5,4,3,2,6)$ \\
			$\longleftrightarrow$& \begin{tabular}{c}
				\begin{tikzpicture}[scale=.5]
					\newcommand\myx[1][(0,0)]{\pic at #1 {myx};}
					\tikzset{myx/.pic = {\draw [ultra thick]
							(-2.5mm,-2.5mm) -- (2.5mm,2.5mm)
							(-2.5mm,2.5mm) -- (2.5mm,-2.5mm);}}
					\draw[gray] (0,0) grid (7,7);
					\foreach \x/\y in {
						1/4,
						2/5,
						3/2,
						4/3,
						5/6,
						6/7,
						7/1
					}
					\myx[(\x-.5,\y-.5)];
					\draw[-, ultra thick,green!70!black] (.5,.5)--(.5,3.5)
					--(3.5,3.5)--(3.5,2.5)
					--(2.5,2.5)--(2.5,1.5)
					--(1.5,1.5)--(1.5,4.5)
					--(4.5,4.5)--(4.5,5.5)
					--(5.5,5.5)--(5.5,6.5)
					--(6.5,6.5)--(6.5,.5)
					--(.5,.5);
			\end{tikzpicture} \end{tabular}&
			\begin{tabular}{c}
				\begin{tikzpicture}[scale=.5]
					\newcommand\myx[1][(0,0)]{\pic at #1 {myx};}
					\tikzset{myx/.pic = {\draw [ultra thick]
							(-2.5mm,-2.5mm) -- (2.5mm,2.5mm)
							(-2.5mm,2.5mm) -- (2.5mm,-2.5mm);}}
					\draw[gray] (0,0) grid (6,6);
					\foreach \x/\y in {
						1/5,
						2/6,
						3/2,
						4/3,
						5/4,
						6/1
					}
					\myx[(\x-.5,\y-.5)];
					\draw[-, ultra thick,red!70!black] (.5,.5)--(.5,4.5)
					--(4.5,4.5)--(4.5,3.5) -- (3.5,3.5)--(3.5,2.5)
					--(2.5,2.5)--(2.5,1.5)
					--(1.5,1.5)--(1.5,5.5)
					--(5.5, 5.5)--(5.5,.5)--(.5,.5);
				\end{tikzpicture}
			\end{tabular}
		\end{tabular}
	}
	\caption{The permutation $\pi\in\A_{11}(4321;213)$ with $\pi_{1}=5$ can be decomposed into a permutation $\pi'~\in~\A_7(4321;231)$ and $\pi''\in\A_6(4321;231)$ with some additional constraints.}
	\label{fig:231-k21-1}
\end{figure}
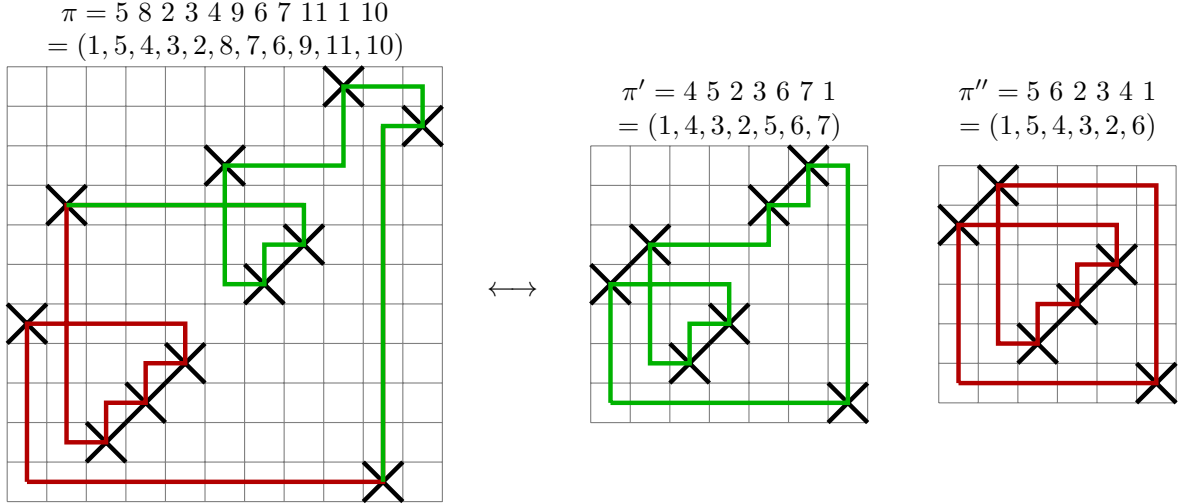

We now consider permutations in $\B_n(\delta_k; 231)$ and enumerate these based on cyclic permutations that avoid 213 in cycle form and $\delta_{k-2}$ in one-line notation.

\begin{lemma}\label{lem:as-k21-2} 
	For $n\geq 3$ and $k\geq 5$, \[b_{n,k}(231) = b_{n-1,k}(231)+ \sum_{j=2}^{n-2} a_{j-1,k-2}(213)a_{n-j,k-2}(213).\]
\end{lemma}

\begin{proof} We will count permutations $\pi \in \B_n(\delta_k; 231)$ based on the value of $\pi_n$. We first show that the number with $\pi_n = n-1$ is $b_{n-1,k}(231).$  Let $\pi \in \B_n(\delta_k; 231)$ with $\pi_n=n-1$. Then
	\[ C(\pi) = (1,n,n-1,c_4, \ldots, c_n).\] Let $\pi'$ be the permutation formed by deleting $n-1$ from $\pi$. Then $\pi'$ clearly avoids $\delta_k$ if and only if $\pi$ does. Also, the cycle notation of $\pi'$ is
	\[ C(\pi') = (1, n-1, c_4, \ldots, c_n)\]
	and thus, $\pi' \in \B_{n-1}(\delta_k; 231)$. Because this process is invertible, there are $b_{n-1, k}$ such permutations as desired.
	
	Now suppose $\pi \in \B_n(\delta_k; 231)$ with $\pi_n=j$ for some $j \in [2,n-2]$. Then
	\[ C(\pi) = (1,n,j,c_4, c_5, \ldots, c_{j+1}, c_{j+2}, \ldots, c_n)\] where $\{c_4, \ldots, c_{j+1}\} = [2, j-1]$ and $\{c_{j+2}, \ldots, c_n\} = [j+1, n-1].$ In one-line form, we then have
	\[ \pi = n\pi_2\pi_3\cdots\pi_j\pi_{j+1}\cdots\pi_{n-1}j\]
	where $\{\pi_2,\ldots,\pi_j\} = [2,j-1]\cup\{c_{j+2}\}$ and $\{\pi_{j+1},\ldots, \pi_{n-1}\} = [j+1,n-1] \setminus \{c_{j+2}\} \cup \{1\}$. We now create two cyclic permutations from $\pi$ that both avoid $\delta_k.$
	
	
	Let $\pi'$ be the permutation formed by deleting the elements in $\{c_{j+3}, c_{j+4}, \ldots, c_n\}$ from $C(\pi).$ Thus,
	\[ C(\pi') = (1, j+2, j, c_4, \ldots, c_{j+1}, j+1),\] and in one-line form
	\[ \pi' = (j+2)\pi_2\pi_3\cdots\pi_{j}1j.\] 
	Notice that $\pi'$ avoids $\delta_k$ and that $C(\pi')$ avoids 231. By Lemma~\ref{lem:231-1}, there are $a_{j-1,k-2}(213)$ such permutations.

	We form the second permutation in the decomposition, ${\pi}''$, by deleting $\{c_4, \ldots, c_{j+1}\}$ from $C(\pi).$ Thus,
	\[ C(\pi'') = (1,n-j+2,2,c_{j+2}-j+2,c_{j+3}-j+2,\ldots, c_n-j+2)\] and
	\[ \pi'' = (n-j+2)(c_{j+2}-j+2)(\pi_{j+1}-j+2)(\pi_{j+2}-j+2)\cdots(\pi_{n-1}-j+2)2.\]
	Notice $\pi''$ must avoid $\delta_k$ because it is a subpattern of $\pi$ which avoids $\delta_k.$ Since this permutation starts with its largest element and ends with 2, by Lemma~\ref{lem:231-1}, we see that the number of such permutations is $a_{n-j,k-2}(213).$
	
	Since, as we stated above, any permutation $\pi$ of this form must look like 
	\[ \pi = n\pi_2\pi_3\cdots\pi_j\pi_{j+1}\cdots\pi_{n-1}j\]
	where $\{\pi_2,\ldots,\pi_j\} = [2,j-1]\cup\{c_{j+2}\}$ and $\{\pi_{j+1},\ldots, \pi_{n-1}\} = [j+1,n-1] \setminus \{c_{j+2}\} \cup \{1\}$, any $\delta_k$ pattern must only consist of elements from either $[1,j-1]\cup\{c_{j+1},n\}$ or from $\{1\}\cup[j,n]$, this implies that doing the inverse process, building a permutation $\pi$ from $\pi'$ and $\pi''$ will result in a permutation that still avoids $\delta_k.$
\end{proof}


Given these results in Lemma~\ref{lem:as-k21-3}, we can define a generating function for $b_{n,k}(231)$ based on the generating function for $a_{n,k}(213).$

\begin{corollary}\label{cor:231} Lt $f_k(z;213)$ be the generating function for $a_{n,k}(213)$ and let $g_k(z;231)$ be the generating function for $b_{n,k}(231)$. Then for $k\geq 5,$
	\[ g_k(z;231) = \frac{z}{1-z}\left[1- zf_{k-2}(z;213) +(f_{k-2}(z;213))^2 \right]\]
\end{corollary}

Using this Corollary along with Lemma~\ref{lem:as-k21-3}, we can now find the generating function for $a_{n,k}(231)$.

\begin{theorem} \label{thm:231} Let $f_k(z;231)$ be the generating function for $a_{n,k}(231).$ Then for $k \geq 5$, $f_k(z;231)$ can be written in terms of the generating functions for $a_{n,k}(213)$:
	\[f_k(z;231) = \frac{f_{k-1}(z;213)\left([f_{k-2}(z;213)]^2-f_{k-2}(z;213)+1\right)}{1-z}. \]
	We also have \[f_4(z;231)=\frac{z (z^3 + z^2 - 2 z + 1)}{(1-z) (1-2z)},\]
	\[f_3(z;231)=\frac{z}{(1-z)^2}\]
	and $f_2(z;231)=f_1(z;231)=z$.
\end{theorem}

\begin{proof} For $k\geq 5$, Corollary~\ref{cor:231} and Lemma~\ref{lem:as-k21-3} combine to give us:
	\[ f_k(z;231) = \frac{z([f_{k-2}(z;213)]^2-f_{k-2}(z;213)+1)}{(1-z)(1-f_{k-2}(z;213))}.\] Using Corollary~\ref{cor:213-k21}, we can simplify this result as seen in the theorem. The case for $k=3$ follows from \cite[Theorem 4.16]{ABBGJ} and the case when $k\leq 2$ is clear. 
	
	All that remains is the case when $k=4.$ Let us first notice that $b_{n,4}(231)=n-2$. Indeed, it is easy to see there are $b_{n-1,4}$ permutations that have $\pi_n=n-1$, and one permutation that has $\pi_n=2$. We need only show that there are no permutations with $\pi_n=j$ for $3\leq j\leq n-2.$ If there were such a permutation, we would have $\pi = (1,n,j,c_3,\ldots, c_{j+1}, c_{j+2},\ldots,c_n)$ with $c_i>j$ for $i\geq j+2$ and $c_i<j$ for $3\leq i\leq j+1$. Thus $\pi_1\pi_{c_{j+1}}\pi_j\pi_{c_n} = nc_{j+2}c_31$ is a 4321 pattern. Next, let us see that any permutation in $\A_{n}(4321;231)$ with $\pi_1=j<n$ must be of the form \[C(\pi) = (1,j,j-1,\ldots,3, 2,c_{j+1},\ldots,c_n)\] where the permutation obtained by removing the elements in $[2,j]$ is in $\A_{n-j+1}(4321;231).$ If this were not the case and $i$ was the first place $c_i<c_{i+1}$ for some $i<j+1$, then $jc_{i+1}c_i1$ would be a 4321 pattern in $\pi.$ Thus we know that \[a_{n,4}(231) = b_{n,4}(231)+\sum_{j=2}^{n-1}a_{n-j+1}(231),\] which together with the fact that $b_{n,4}(231)=n-2$ implies the given generating function for $k=4$.
\end{proof}


As an example, consider $k=5$. We have
\begin{align*} f_5(z;231) &= \frac{f_{4}(z;213)\left([f_{3}(z;213)]^2-f_{3}(z;213)+1\right)}{1-z}\\
	&= \frac{z(1-2z)}{(z^2-3z+1)(1-z)} - \frac{z^2}{1-2z}.
\end{align*}

\section{Avoidance in all cycle forms}\label{sec:allcycles}

In this section, we will consider when the one-line notation of a cyclic permutation $\pi$ avoids $\delta_k=k(k-1)\ldots 21$ and \emph{all} cycle forms (that is, all cyclic rotations of $C(\pi)$) avoid a given pattern $\tau.$
As noted in \cite{Callan,DELMSSS, V03}, the case where $\tau$ is of length 3 is trivial, and so here, we consider $\tau\in\S_4$. In particular, we consider the case $\tau\in\{1324,1342\}$. Up to symmetry, the only remaining case to consider would be $\tau=1234,$ which we leave as an open question.

In this section, we use $\A^\circ_n(\delta_k;\tau)$ to denote those cyclic permutations that avoid $\delta_k$ where all cyclic rotations of $C(\pi)$ avoids $\tau$, and we let $a_{n,k}^\circ(\tau)=|\A^\circ_n(\delta_k;\tau)|$.

\subsection{Enumerating $\A^\circ_n(\delta^k;1324)$ and $\A^\circ_n(\delta^k;1423)$}

In \cite{Callan, V03}, the authors find the the total number of cyclic permutations $\pi$ so that all rotations of $C(\pi)$ avoid 1342 is equal to $F_{2n-3}$, the $(2n-3)^{\text{rd}}$ Fibonacci number.

\begin{theorem}\label{thm:1324}
	For $n\geq 3$, $a^\circ_{n,3}(1324) = 2^{n-2}$ and for $k\geq 4$, $a^\circ_{n,k}(1324) = F_{2n-3}$.
\end{theorem}

\begin{proof}
	Let us first notice that if $\pi\in\A^\circ_n(\delta_k;1324)$
	then we must have $C(\pi) = (1, c_2, \ldots, c_{r-1}, 2, c_{r+1}, \ldots, c_n)$ with $\{c_2,\ldots,c_{r-1}\}=[n-r+3,n]$ and $\{c_{r+1},\ldots,c_n\}=[3,n-r+2].$ Furthermore, if $\pi_1\neq 2$, we must have that the elements after 2 appear in increasing order since otherwise we would have a 1324 pattern in the cycle form that begins with 2, with $2$ acting as the 1 in the pattern and $c_2$ acting as the 4 in the pattern.
	
	Now, if $k=3$ we claim that we must have that $\pi_1=2$ or $\pi_2=1$. If not, then we have $C(\pi) = (1, c_2, \ldots, c_{r-1}, 2, 3, \ldots, n-r+2)$ for some $r$. Then, in $\pi$, we have the 321 pattern $\pi_1\pi_2\pi_{n-r+2}=\pi_131$ with $\pi_1>2$. It is also straightforward to see that for any permutation $\pi \in \A^\circ_n(321;1324)$ with $C(\pi) = (1,2,c_3, \ldots, c_n)$, the permutation $\pi'$ with $C(\pi') = (1,c_3-1, \ldots, c_n-1)$ is in $\A^\circ_{n-1}(321; 1324)$ and in fact any permutation in $\A^\circ_{n-1}(321; 1324)$ can be obtained this way. Similarly if $\pi$ satisfies $C(\pi) = (1,c_2, \ldots, c_{n-1},2)$, then the permutation $\pi'$ with $C(\pi') = (1,c_2-1, \ldots, c_{n-1}-1)$ is in $\A^\circ_{n-1}(321, 1324)$. Together with the fact that $a_2^\circ(321,1324)=1$, the result that $a^\circ_{n,3}(1324) = 2^{n-2}$ follows. 
	
	For the second part of the theorem, it is enough to show that if all cyclic rotations of $C(\pi)$ avoid the pattern $1324$, then $\pi$ avoids 4321. As noted above, all permutations in $\A^\circ_{n}(\delta_k;1324)$ must either have the property that $\pi_1=2$ or $C(\pi) = (1,c_2,\ldots, c_{n+1-k},2,3,\ldots,k)$ for some $k>2$ where the permutation obtained by deleting $[2,k]$ from $C(\pi)$ (equivalently, deleting $34\ldots k1$ from positions $\{2,\ldots,k\}$ of the one-line notation) is also in $\A^\circ_{n-k+1}(\delta_k;1324)$. Additionally, if $\pi_1=2,$ deleting 2 from the cycle notation (equivalently, deleting 2 from the front of one-line notation), leaves you with a permutation in $\A^\circ_{n-1}(\delta_k;1324)$. Since all the permutations in $\A^\circ_{n}(\delta_k;1324)$ are built recursively by adding 2 to the front of the one-line notation or $34\ldots k1$ after the first position of the one-line notation, it is impossible to introduce a new 4321 pattern to the one-line notation, and so all permutations whose cycle forms avoid 1324 must avoids 4321 in its one-line notation.
\end{proof}

Finally, by considering the reverse of 1342, which is equivalent to 1423 under cyclic rotation, we have the following corollary.

\begin{corollary}
	For $n\geq 3$, $a^\circ_{n,3}(1423) = 2^{n-2}$. For $k\geq 4$, $a^\circ_{n,k}(1423) = F_{2n-3}$. 
\end{corollary}

\subsection{Enumerating $\A^\circ_n(\delta^k;1342)$}

In this final section, we consider the case where $\tau=1342.$
In \cite{Callan, V03}, the authors find the the total number of cyclic permutations $\pi$ so that all rotations of $C(\pi)$ avoid 1342 is equal to $2^{n-1}-n+1$.

\begin{theorem}\label{thm:1342}
	If $f^\circ_k(z;1342) = \sum_{n\geq 0}a_{n,k}^\circ(1342)z^n$, then $f_1^\circ(1342)=f_2^\circ(1342) = 1+z$, and for $k\geq 3,$
	\[f^\circ_k(z;1342) = \frac{1-3z+2z^2+z^3}{(1-z)^2(1-2z)} - \frac{2z^{k+1}}{(1-z)^{k-1}(1-2z)}.\]
\end{theorem}

Let us first address the case when $k=3$, and then we will find a recurrence for the cases when $k\geq 4.$

\begin{lemma}
	For $n\geq 3,$ $a^\circ_{n,k}(1342) = n-1.$
\end{lemma}
\begin{proof}
	First, note that if $C(\pi) = (1,c_2,\ldots, c_{k-1},n,c_{k+1},\ldots,c_n)$, then since $C(\pi)$ avoids 1342, we must have that $\{c_2,\ldots,c_{k-1}\}=[2,k-1]$ and $\{c_{k+1},\ldots,c_n\}=[k,n-1]$. That means $\pi = \pi_1\pi_2\ldots\pi_n$ satisfies that $\{\pi_1,\ldots,\pi_{k-1}\}=[2,k-1]\cup\{n\}$ and $\{\pi_k,\ldots,\pi_n\}=\{1\}\cup[k,n-1]$. Since $\pi$ avoids 321, we must have that $\pi_{k-1}=n$ and $\pi_k=1$ or else either $n\pi_{k-1}1$ or $n\pi_k1$ would be a 321 pattern. However, this implies that the elements appearing before $1$ and those appearing after $n$ must be increasing, so $\pi=23\ldots(k-1)n1k(k+1)\ldots(n-1).$ Since all cyclic rotations of the cycle form $C(\pi) = (1,2,\ldots,k-1,n,n-1,\ldots,k+1,k)$ avoid 1342, there are $n-1$ permutations in $\A^\circ_n(\delta_k;1342).$
\end{proof}

In order to avoid $\delta_k$ for $k\geq 4$, we will establish a recurrence. As before, $\B_n^\circ(\delta_k;1342)$ are those permutations in $\A_n^\circ(\delta_k;1342)$ with $\pi_1=n$ and $b^\circ_{n,k}(1342)=|\B_n^\circ(\delta_k;1342)|$.

\begin{lemma}\label{lem:1342-2}
	For $n\geq 3$ and $k\geq 4$, the number of permutations in $\A^\circ_n(\delta_k;1342)$ with $\pi_1=n-1$ and $\pi_n=1$ is equal to $b^\circ_{n-1,k-1}(1342)$.
\end{lemma}

\begin{proof}
	Given $\A^\circ_n(\delta_k;1342)$ with $\pi_1=n-1$, $\pi_n=1$, and $\pi_j=n$, let $\pi'\in \S_n$ is defined by letting $\pi'_j=1$ and $\pi'_i=\pi_i$ for $i\in[n-1]\setminus\{j\}$. This corresponds to deleting $n$ from the cycle form of $\pi$ to obtain $\pi'$. We claim that the permutation $\pi'\in\B_{n-1}^\circ(\delta_{k-1};1342)$ and that all such permutations in $B_{n-1}^\circ(\delta_{k-1};1342)$ can be obtained this way. 
	
	First notice that if there is a $\delta_{k-1}$ pattern in $\pi'$, that pattern together with the 1 at the end of $\pi$ would be a $\delta_k$ pattern in $\pi.$ Conversely, a $\delta_k$ pattern in $\pi$, then there is one that does not use $n$ as the $k$ in the $k\ldots 21$ pattern (since we can use $\pi_1=n-1$ instead). By taking the first $k-1$ elements of that pattern (which must not contain $\pi_n=1$) we would have a $\delta_{k-1}$ pattern in $\pi'.$  
	
	Finally, since adding or removing the $n$ at the end (which is cyclically adjacent to $1,n-1$), we cannot introduce a new $1342$ pattern to any cyclic rotation of the cycle forms of the permutations and so the result follows.
\end{proof}

\begin{lemma}\label{lem:1342-3}
	For $n\geq 5$ and $k\geq 4$, we have 
	\[
	a^\circ_{n,k}(1342) = a^\circ_{n-1,k}(1342) + b^\circ_{n,k}(1342)+ \sum_{r=3}^{n-1} b^\circ_{r,k-1}(1342) .
	\]
\end{lemma}

\begin{proof}
	Let $\pi\in\A_n^\circ(\delta_k;1342)$.
	Note that if $\pi_1=2$, one can delete the 2 from both the cycle from and one-line form of the permutation and get a new permutation in $\pi\in\A_{n-1}^\circ(\delta_k;1342)$ and that this process is invertible since that 2 cannot be part of a $\delta_k$ pattern in $\pi$ and cannot add a $1342$  pattern in any cyclic rotation of $C(\pi)$. Also, by definition there are $b^\circ_{n,k}(1342)$ permutations with $\pi_1=n.$ 
	
	Assume that $3\leq \pi_1 \leq n-1$, and let $r=\pi_1.$ Then we claim that the permutation must be of the form $(1,r,c_3,\ldots, c_r,n,n-1,\ldots, r+1).$ Indeed, since $\pi$ avoids $1342$, all elements after $n$ must be larger than elements before $n$. Furthermore, if there were some element $c_m<c_{m+1}$ with $m>r$, then $c_3c_mc_{m+1}r$ would be a 1342 pattern in a cyclic rotation of $C(\pi)$. Therefore, in one-line notation, $\pi$ is of the form $\pi=r\pi_2\ldots\pi_r1(r+1)(r+2)\ldots(n-1)$. Any $\delta_k$ pattern in $\pi$ does not include any element of the increasing segment $(r+1)(r+2)\ldots(n-1)$ since $n$ is the only element larger than those appearing before them. For this reason, the permutation obtained by deleting these elements (from both the one-line and cycle form of the permutation) still avoids is of the form  $(1,r,c_3,\ldots,c_r,r+1)$ and avoids $\delta_{k}$ and that all such cycles can be formed this way. Thus, by Lemma~\ref{lem:1342-2}, there are $b^\circ_{r,k-1}$ such permutations with $\pi_1=r$. 
\end{proof}

\begin{lemma}\label{lem:1342-4} For $n\geq 1$, we have $b^\circ_{n,3}(1342) = 1$. For $n\geq 4$ and $k\geq 4$, we have 
	\[
	b^\circ_{n,k}(1342) = b^\circ_{n-1,k}(1342) + b^\circ_{n-1,k-1}(1342).
	\]
	As a result we have the generating function $B_k(z;1342) = \sum_{n\geq 1} b^\circ_{n,k}(1342)z^n$
	\[
	B_k(z;1342) =\frac{z (1 - z - z^2)}{1 - 2 z}- \frac{z^{k+1}}{(1 - 2 z)(1 - z)^{k-2}}.
	\]
\end{lemma}

\begin{proof}
	We first claim that if $\pi\in\B_n^\circ(\delta_k;1342)$, then we must have $\pi_n\in\{2,n-1\}.$ If not, then we would have $C(\pi) =(1,n,j,c_4,\ldots,c_n)$ for some $3\leq j\leq n-2$. Notice that if $2$ appears after $n-1$ in $C(\pi)$, then $1j(n-1)2$ is a 1342 pattern, and if $2$ appears before $n-1$, then $2(n-1)nj$ is a 1342 pattern in a cyclic rotation of $C(\pi).$ Since this is a contradiction, we must have $\pi_n\in\{2,n-1\}.$ If $\pi_n=2$, then by complementing $C(\pi)$ (corresponding to the reverse complement of $\pi$), we have a permutation $\pi^{rc}\in\A_n^\circ(\delta_k;1342)$ with the property that $\pi_1=n-1$ and $\pi_n=1$. By Lemma~\ref{lem:1342-2}, there are $b_{n-1,k-1}^\circ(1342)$ such permutations. If $\pi_n=n-1$, we can delete $n-1$ from the cycle form (and equivalently from the one-line notation) to obtain a permutation in $\B_n^\circ(\delta_k;1342)$. Since this process is invertible, there are $b_{n-1,k}^\circ(1342)$ permutations with $\pi_n=n-1$. 
	
	Finally, since the generating function in the statement of the theorem satisfies \[B_k(z;1342) = zB_{k}(z;1342) + zB_{k-1}(z;1342) + z-z^2-z^3,\] the theorem follows. 
\end{proof}

\begin{proof}[Proof of Theorem~\ref{thm:1342}]
	The generating function $f_k^\circ(z;1342)$ in the statement of the theorem follows from the recurrence in Lemma~\ref{lem:1342-3} and the generating function in Lemma~\ref{lem:1342-4}.
\end{proof}


\section{Open Questions}\label{sec:open}

In Section~\ref{sec:standard}, we enumerate cyclic permutations that avoid the decreasing pattern $\delta_k=k(k-1)\ldots 21$ and whose cycle form $C(\pi)$ avoids a pattern $\tau$ of length 3, with the exception of $\tau=321,$ which remains an open question. Similarly, in Section~\ref{sec:allcycles}, we enumerate cyclic permutations that avoid the decreasing pattern $\delta_k=k(k-1)\ldots 21$ such that every cyclic rotation of its cycle form avoids a pattern $\tau$ of length 4, with the exception of $\tau=1432,$ which also remains open. 

There are many other interesting sequences that seem to appear when considering pattern-avoiding cyclic permutations where the cycle form(s) also avoids a pattern or set of patterns. For example, we conjecture that $|\A^\circ(4123;1324)|$ is the $(n-2)^{\text{nd}}$ Tetranacci number, that $|\A^\circ(2431;1324)|$ is the $(n-1)^{\text{st}}$ Pell number, and $|\A^\circ(4132;1324)|$ is the $(3n)^{\text{th}}$ Padovan number. Several other pairs of patterns (or sets of patterns) seem to give interesting sequences that appear in the Online Encyclopedia of Integer Sequences, which makes these open questions seem quite accessible to future researchers.

\subsection*{Acknowledgements}

The student authors on this paper, Ethan Borsh, Jensen Bridges, and Millie Jeske, were funded as part of an REU at the University of Texas at Tyler sponsored by the NSF Grant DMS-2149921.

\bibliographystyle{amsplain}


\end{document}